\documentclass{article}%

\usepackage{amsmath,amssymb,amsfonts}%
\usepackage{theorem}%
\usepackage{color}%
\usepackage{bbm}%

\theoremstyle{change}%
\newtheorem{definition}{Definition:}[section]%
\newtheorem{proposition}[definition]{Proposition:}%
\newtheorem{theorem}[definition]{Theorem:}%
\newtheorem{lemma}[definition]{Lemma:}%
\newtheorem{corollary}[definition]{Corollary:}%
{\theorembodyfont{\rmfamily}\newtheorem{remark}[definition]{Remark:}}%
{\theorembodyfont{\rmfamily}}%

\newenvironment{proof}
  {{\bf Proof:}}
  {\qquad \hspace*{\fill} $\Box$}%

\newcommand{\fg}{\mathfrak{g}}%
\newcommand{\fh}{\mathfrak{h}}%
\newcommand{\rme}{\mathrm{e}}%
\newcommand{\DC}{\mathcal{D}}%
\newcommand{\XC}{\mathcal{X}}%
\newcommand{\FC}{\mathcal{F}}%
\newcommand{\YC}{\mathcal{Y}}%

\newcommand{\N}{\mathbb{N}}%
\newcommand{\R}{\mathbb{R}}%

\begin{document}

\title{Topological conjugacy of linear systems on Lie groups\thanks{This work was partially supported by
CNPq/Universal grant n$^{\circ }$ 476024/2012-9}}
\author{A. Da Silva\thanks{Supported by CAPES grant
n$^{\circ }$ $4792/2016$-PRO and partially supported by FAPESP grant
n$^{\circ }$ $2013/19756-8$}\\Departamento de Matem\'{a}tica, Universidade
de Campinas - Brazil, and \\
Departamento de Matem\'{a}tica, Universidade
Estadual de Maring\'a - Brazil,\\ ajsilva@ime.unicamp.br \and A. J.
Santana\thanks{Partially  supported by Funda\c{c}\~{a}o
Arauc\'{a}ria grant n$^{\circ }$ $20134003$}\\Departamento de
Matem\'{a}tica, Universidade Estadual de Maring\'{a}\\Maring\'{a},
Brazil, ajsantana@uem.br\and S. N. Stelmastchuk\\ Universidade
Federal do Paran\'{a}\\Jand\'{a}ia do Sul, Brazil,
simnaos@gmail.com}

\maketitle

\begin{abstract}
In this paper we study a classification of linear systems on Lie groups with respect to the conjugacy of the corresponding flows. We also describe  stability according to Lyapunov exponents.
\end{abstract}



\section{Introduction}

Consider the linear control system in $\R^n$ given by
\begin{equation}\label{linearsystemRn}
  \dot{x} = Ax + \sum_{i=1}^nu_ib_i,
\end{equation}
where $A$ is a $n\times n$ matrix, $(b_1,\ldots,b_n)$ is a vector in $\R^n$ and $u(t) =(u_1(t),\ldots, u_n(t))$ is locally integrable function. This kind of the linear control system has been extensively studied as it can be viewed in Agrachev and Sachkov \cite{agrachev}, Jurdjevic \cite{jurdjevic} and Sontag \cite{sontag}, and the references therein. Commonly, three aspects of control system are studied: controllability, classification and optimality. Our first wish is to study the classification of the linear control systems on connected Lie groups, more specifically, the topological conjugacy. 

Here we follow the work of Ayala and Tirao \cite{ayalatirao}. Inspired by a work due to Markus \cite{markus}, Ayala and Tirao introduced the concept of linear control systems on connected Lie groups. They shown that the generalization of the linear control system (\ref{linearsystemRn}) to a connected Lie group $G$ is determined by the family of differential equations:  
\begin{equation}  \label{contsyst1}
   \dot{g}(t)={\mathcal{X}}(g(t)) + \sum_{i=1}^n u_i(t) X_i(g(t)),
\end{equation}
where $\mathcal{X}$ is an infinitesimal automorphism of $G$, $X_i$ are right invariant vector fields on $G$ and $u(t)=(u_1(t), \ldots, u_n(t)) \in \mathbb{R}^n$ belongs to the class of unrestricted admissible control functions. It is worth note that there are several works about the control system (\ref{contsyst1}) when the drift $\mathcal{X}$ is assumed to be a right invariant vector field. The reader interested can found more about this case, for instance, in the work of Biggs and Remsing \cite{biggsremsing} and the references therein.  

To the linear control system (\ref{contsyst1}) there exist a number of works concerning the controllability problem. Ayala and Tirao studied in \cite{ayalatirao} local controllability problems and the (ad)-rank conditions.  After, Ayala and San Martin \cite{ayalasanmartin} establish controllability results for compact and semisimple Lie groups. In sequence, Jouan \cite{jouan1}, \cite{jouan2} and Jouan and Dath \cite{jouandath} studied about equivalence and controllability of linear control systems. In \cite{jouan1}, the importance of the linear control systems on Lie groups is once more highlighted. It is shown that the class of linear control systems classifies several classes of affine control systems on arbitrary connected manifolds.

Recently, Da Silva \cite{dasilva} and Ayala and Da Silva \cite{ayaladasilva} shown that the dynamics of the flow associated with the drift are intrinsically connected with the behavior of the whole linear control system. They shown that, if the flow associated with the drift has trivial expanding or contracting subgroups, then the control system (\ref{contsyst1}) is controllable.  

Therefore, to know the dynamical properties of such flows is fundamental in order to understand the behavior of the linear control systems. Observing this fact, before studying the topological conjugacy between system of type (\ref{contsyst1}) we view that it is necessary to study the topological conjugacy in a intermediary step, that is, we adopt the linear system on a connected Lie group $G$ given by 
\begin{equation}  \label{syst1}
  \dot{g}(t)={\mathcal{X}}(g(t)),
\end{equation}
 where the drift ${\mathcal{X}}$ is an infinitesimal automorphism. Thus, the main purpose of this paper is to classify linear flows on Lie group
via topological conjugacies and  characterize asymptotic and exponential stability  using Lyapunov exponents. Besides this purpose our work reveals the similarity between (\ref{syst1}) and the linear system in $\mathbb{R}^n$ given by 
\begin{equation}\label{linearsystem}
   \dot{x} = Ax.
\end{equation}
Then following a similar approach of the classical result we classify the linear flows according to the decomposition of the state space in stable, unstable and central Lie subgroups. Moreover, given a fixed point $g \in G$ of the linear flow given by
${\mathcal{X}}$, we define when $g$ is stable, asymptotically stable and exponentially stable. From this we characterize these stabilities according to Lyapunov exponents.

The paper is organized as follows, in the second section we establish some results,  prove that the stable and unstable subgroups are simply connected and characterize the stable and unstable elements as attractors and repellers. In the third section we prove an important result of the paper, that is, given two linear vector fields we give conditions on it and on their stable spaces in order to ensure that their respective flows are conjugated. Finally, in the last section we study the Lyapunov stability.

\section{Stable and unstable Lie subgroups}

In this section we present   notations and basic tools to prove that the stable, unstable components of the state space are simply connected. Moreover in the main result of this section we characterize the stable and unstable elements   as attractors and repellers respectively.

\begin{definition}
  Let $\left\{\varphi_t\right\}_{t\in\R}$ be a flow of automorphisms on a connected Lie group $G$. We say that $\varphi_t$ is {\bf contracting} if there are constants $c, \mu>0$ such that
\[
  \left|\left|(d\varphi_t)_eX\right|\right|\leq c\rme^{\mu t}||X||\;\;\;\;\mbox{ for any }X\in\fg;
\]
We say that $\varphi_t$ is {\bf expanding} if $\left(\varphi_t\right)^{-1}=\varphi_{-t}$ is contracting.
\end{definition}

The next results states that the existence of contracting/expanding flow of automorphisms on a Lie group requires some topological properties.

\begin{proposition}\label{contracting}
Let $G$ be a connected Lie group and $\left(\varphi_t\right)$ a flow of automorphism on $G$. If $\varphi_t$ is contracting or expanding, the Lie group $G$ is simply connected.
\end{proposition}

\begin{proof}
Let us assume that $\varphi_t$ is contracting, since the expanding case is analogous. Let $\widetilde{G}$ be the connected simply connected cover of $G$ and consider $D$ to be the central discrete subgroup of $\widetilde{G}$ such that $G=\widetilde{G}/D$.

Since the canonical projection $\pi:\widetilde{G}\rightarrow G$ is a covering map and $\widetilde{G}$ is simply connected, we can lift $\left(\varphi_t\right)_{t\in\R}$ to a flow $\left(\widetilde{\varphi}_t\right)_{t\in\R}$ on $\widetilde{G}$ such that
$$\pi\circ \widetilde{\varphi}_t=\varphi_t\circ\pi, \;\;\mbox{ for all }\;\; t\in\R.$$ 
In particular $\widetilde{\varphi}_t(D)=D$ for all $t\in\R$.  Being that $D$ is discrete and $\widetilde{\varphi}_t$ is continuous, we must have that $\widetilde{\varphi}_t(x)=x$ for all $x\in D$ and $t\in\R$. However, since $\varphi_t$ is contracting, $\widetilde{\varphi}_t$ is also contracting and so, $\widetilde{\varphi}_t(x)\rightarrow \tilde{e}$ when $t\rightarrow\infty$, where $\tilde{e}\in\widetilde{G}$ is the identity element. Being that $D$ is discrete and $\widetilde{\varphi}_t$-invariant, we must have for $x\in D$ and $t>0$ large enough that $\widetilde{\varphi}_t(x)=e$ which implies that $x=e$. Therefore, $D=\{e\}$ and so $\widetilde{G}=G$, which concludes the proof.
\end{proof}

\begin{corollary}\label{compacityproperty}
If $G$ is a compact Lie group it admits no expanding or contracting flow of automorphisms.
\end{corollary}

Now we can prove that unstable and stable subgroups of the linear flow are simply connected.  Let us assume that $G$ is a connected Lie group and let $\XC$ be a linear vector field on $G$ with linear flow $(\varphi_t)_{t\in\R}$. Associated with $\XC$ we have the connected $\varphi$-invariant Lie subgroups $G^+$, $G^0$ and $G^-$
with Lie algebras given, respectively, by
$$
  \fg^+=\bigoplus_{\alpha; \mathrm{Re}(\alpha)>0}\fg_{\alpha}, \;\;\;\;\fg^0=\bigoplus_{\alpha; \mathrm{Re}(\alpha)=0}\fg_{\alpha}, \;\;\;\mbox{ and }\;\;\;\fg^-=\bigoplus_{\alpha; \mathrm{Re}(\alpha)<0}\fg_{\alpha},
$$
where $\alpha$ is an eigenvalue of the derivation $\DC$ associated with $\XC$ and $\fg_{\alpha}$ its generalized eigenspace.

We consider also $G^{+, 0}$ and $G^{-, 0}$ as the connected $\varphi$-invariant Lie subgroups with Lie algebras $\fg^{+, 0}=\fg^+\oplus\fg^0$ and $\fg^{-, 0}=\fg^-\oplus\fg^0$ respectively. The next proposition states the main properties of the above subgroups, its proof can be found in \cite{dasilva} and \cite{DaAy2}.

\begin{proposition}
\label{subgroups}
It holds:
\begin{itemize}
\item[1.] $G^{+, 0}=G^+G^0=G^0G^+$ and $G^{-, 0}=G^-G^0=G^0G^-$;
\item[2.] $G^+\cap G^-=G^{+, 0}\cap G^-=G^{-, 0}\cap G^+=\{e\}$;
\item[3.] $G^{+, 0}\cap G^{-, 0}=G^0$;
\item[4.] All the above subgroups are closed in $G$;
\item[5.] If $G$ is solvable then
\begin{equation}
\label{decomposition}
G=G^{+, 0}G^-=G^{-, 0}G^+
\end{equation}
Moreover, the singularities of $\XC$ are in $G^0$;
\item[6.] If $\DC$ is inner and  $G^0$ is compact then $G=G^0$. Moreover, if $G^0$ is compact, then $G$ has the decomposition (\ref{decomposition}) above.
\end{itemize}
\end{proposition}

The subgroups $G^+$, $G^-$ are called, respectively, the {\bf unstable} and {\bf stable subgroups }of the linear flow $\varphi_t$. We denote the restriction of $\varphi_t$ to $G^+$ and $G^-$, respectively, by $\varphi_t^+$ and $\varphi_t^-$. The next proposition gives us another topological property of such subgroups.

\begin{proposition}
\label{simplyconnected}
The Lie subgroups $G^+$ and $G^-$ are simply connected.
\end{proposition}

\begin{proof}
Since $(d\varphi_t)_e=\mathrm{e}^{\DC}$ restricted to $\fg^+$ and to $\fg^-$ has only eigenvalues with positive and negative real parts, respectively, we have that $\left(\varphi_t^+\right)_{t\in\R}$ is an expanding flow of automorphisms and $\left(\varphi^-_t\right)_{t\in\R}$ a contracting flow of automorphisms. Proposition \ref{contracting} implies that $G^+$ and $G^-$ are simply connected.
\end{proof}

Next we will prove a technical lemma that will be needed in the next sections.

\begin{lemma}
\label{convergence}
Let us assume that $G$ is a connected Lie group and let $H_1, H_2$ closed subgroups of $G$ such that $G=H_1H_2$ and $H_1\cap H_2=\{e\}$. For a given sequence $(x_n)$ in $G$ consider the unique sequences $(h_{1, n})$ in $H_1$ and $(h_{2, n})$ in $H_2$ such that $x_n=h_{1, n}h_{2, n}$. Then, $x_n\rightarrow x$ if and only if $h_{i, n}\rightarrow h_i$ for $i=1, 2$ where $x=h_1h_2$.
\end{lemma}

\begin{proof}
Certainly if $h_{i, n}\rightarrow h_i$ in $G$, $i=1, 2$ we have by the continuity of the product of the Lie group $G$ that $x_n=h_{1, n}h_{2, n}\rightarrow h_1h_2=x$ in $G$. Let us assume then that $x_n\rightarrow x$. By considering $\tilde{h}_{1, n}=h_1^{-1}h_{1, n}$ and $\tilde{h}_{2, n}=h_{2, n}h_2^{-1}$ we can assume that $x_n\rightarrow e$ and we need to show that $h_{i, n}\rightarrow e$, $i=1, 2$. Let $U_i$ be a neighborhood of $e\in H_i$, $i=1, 2$. By the conditions on $H_i$, there are neighborhoods $V_i\subset U_i$ of $e\in H_i$, $i=1,2$ such that $W=V_1V_2$ is an open set of $G$ (see Lemma 6.14 of \cite{SM2}) and being that $x_n\rightarrow e$ there is $N\in \N$ such that for $n\geq N$ we have $x_n\in W$ which by the condition that $H_1\cap H_2=\{e\}$ implies $h_{i, 2}\in V_i\subset U_i$, $i=1, 2$ showing that $h_{i, n}\rightarrow e$ for $i=1, 2$ as desired.
\end{proof}

Next we characterize the stable and unstable elements  as attractors and repellers respectively. First we introduce the concept of hyperbolic linear vector field.

\begin{definition}
  Let $\mathcal{X}$ be a linear vector field. We say that $\XC$ is {\bf hyperbolic} if its associated derivation $\DC$ is hyperbolic, that is, $\DC$ has no eigenvalues with zero real part.
\end{definition}

\begin{remark}
  If $\fg_{\alpha}$ is the generalized eigenspace associated with the eigenvalue $\alpha$ of $\DC$, it is well known that 
	$$
	  [\fg_{\alpha}, \fg_{\beta}]\subset\fg_{\alpha+\beta}
	$$ 
	when $\alpha+\beta$ is an eigenvalue of $\DC$ and zero otherwise (see for instance Proposition 3.1 of \cite{SM1}). This implies, in particular, that a necessary condition for the existence of a hyperbolic linear vector field on a Lie group $G$ with $\dim G<\infty$ is that $G$ is a nilpotent Lie group. 
\end{remark}

To define the attractors and repellers we need on Lie Group $G$ a metric space structure. Let $\varrho$ stands for a left invariant Riemmanian distance on $G$. Using that $\varphi_t\circ L_g=L_{\varphi_t(g)}\circ\varphi_t$ we get
\[
  \varrho(\varphi_t(g), \varphi_t(h))\leq ||(d\varphi_t)_e||\varrho(g, h), \;\;\;\mbox{ for any }\;\;g, h\in G.
\]
In particular, since $(d\varphi^-_t)_e=\rme^{t\DC|_{\fg^-}}$ has only eigenvalues with negative real part, there are constants $c, \mu>0$ such that
\begin{equation}\label{metriccontract}
  \varrho(\varphi^-_t(g), \varphi^-_t(h))\leq c^{-1}\rme^{-\mu t}\varrho(g, h), \;\;\;\mbox{ for any }\;\;g, h\in G^-, t\geq 0.
\end{equation}
Analogously, we have that
\begin{equation}\label{metricexpand}
  \varrho(\varphi^+_t(g), \varphi^+_t(h))\geq c\rme^{\mu t}\varrho(g, h), \;\;\;\mbox{ for any }\;\;g, h\in G^+, t\geq 0.
\end{equation}

Now we have the main result of this section.
\begin{theorem}\label{metriccharacterization}
  If $\XC$ is a hyperbolic linear vector field on $G$ then
  \begin{enumerate}
    \item It holds that $g \in G^{-}$ if and only if $\lim_{t \to \infty} \varphi_t(g) =e$.
    \item It holds that $g \in G^{+}$ if and only if $\lim_{t \to - \infty} \varphi_t(g) = e$.
  \end{enumerate}
\end{theorem}
\begin{proof}
  We will prove only the first assertion, since the proof of second assertion is analogous. Let $g \in G^{-}$. It is clear that $\varphi_t(g) = \varphi_t^-(g)$. From equation (\ref{metriccontract}) it follows that
  \begin{eqnarray*}
    \varrho(\varphi^-_t(g), e) = \varrho(\varphi^-_t(g), \varphi^-_t(e))\leq c^{-1}\rme^{-\mu t}\varrho(g, e).
  \end{eqnarray*}
  Taking $t \to \infty$ we get $ \varrho(\varphi^-_t(g), e) \to 0$, that is, $\lim_{t \to \infty} \varphi_t(g) =e$.

  Conversely, let $g\in G$ and assume that $\lim_{t \to \infty} \varphi_t(g) = e$. Being that $\XC$ is hyperbolic, $G$ is nilpotent and consequently, by Proposition \ref{subgroups}, we have that $g$ can be written uniquely as $g=g^-g^+$ with $g^{\pm}\in G^{\pm}$. By the left invariance of the metric we have that
  $$\varrho(\varphi_t(g^+), e)=\varrho(\varphi_t(g^-)\varphi_t(g^+), \varphi_t(g^-))=\varrho(\varphi_t(g), \varphi_t(g^-))$$
  $$\leq \varrho(\varphi_t(g), e)+\varrho(e, \varphi_t(g^-)).$$
 Since $g^-\in G^-$ we have by the first part of the proof and by our assumption that
 $$\lim_{t \to \infty} \varphi_t(g^-)=\lim_{t \to \infty} \varphi_t(g) =e$$
 implying that $\lim_{t \to \infty} \varphi_t(g^+) =e$. This together with the inequality (\ref{metricexpand}) implies that $ g^+=e$ and consequently that $g=g^-\in G^-$ as desired.
\end{proof}

\begin{remark}
  It is well know that the solution to the linear system (\ref{linearsystem}) in $\mathbb{R}^n$ with initial condition $x_0$ is $e^{At}x_0$. Furthermore, if we consider the Euclidian metric, then a version of Theorem \ref{metriccharacterization} is easily obtained. In fact, the definition of hyperbolic system in $\mathbb{R}^n$ gives this result trivially (see for example \cite{robinson}).
\end{remark}

\section{Conjugation between linear flows}

In this section we classify the linear vector fields based on topological conjugacies between their associated flows. From now on we will consider $\XC$ and $\YC$ to be linear vector fields on connected Lie groups $G$ and $H$, respectively, and denote their linear flows by $(\varphi_t)_{t\in\R}$ and $(\psi_t)_{t\in\R}$ and their associated derivation by $\DC$ and $\FC$, respectively. We say that $\XC$ and $\YC$ are topological conjugated if there exists a homeomorphism $\pi: G\rightarrow H$ that commutates $\varphi_t$ and $\psi_t$, that is, 
\[
  \pi(\varphi_t(g))=\psi_t(\pi(g)), \;\;\;\mbox{ for any }\;\;g\in G.
\]
The next result establishes a first conjugation property of these restrictions.

\begin{lemma}\label{lemadimensao}
  It holds that $\dim \fg^+ = \dim \fh^{+}$ if and only if there exists a homeomorphism $\xi:\fg^+\rightarrow\fh^+$ such that
  \[
    \xi(\rme^{t\DC^+}X)=\rme^{t\FC^+}\xi(X), \;\;\;\mbox{ for any }\;\;X\in\fg^+,
  \]
  where $\DC^+$ and $\FC^+$ are the restrictions of $\DC$ and $\FC$ to $\fg^+$ and $\fh^+$, respectively. Analogously, the same is true if $\dim \fg^{-} = \dim \fh^{-}$.
\end{lemma}
\begin{proof}
   We begin by writing $\dim\fg^+ = n$ and $\dim \fh^+=m$. Then there are two isomorphism $S: \fg^+ \rightarrow \R^n$ and $T: \fh^+ \rightarrow \R^m$. Thus we define   the linear maps $\widetilde{\mathcal{D}}: \R^n \rightarrow \R^n$ by $\widetilde{\DC} =  S \DC^+ S^{-1}$ and $\widetilde{\FC}: \R^m \rightarrow \R^m$ by $\widetilde{\FC} =  T \FC^+ T^{-1}$. We claim that eigenvalues of $\DC^+$ and $\widetilde{\DC}$ are the same. In fact, it is sufficient to view the characteristic polynomials
  \[
    \det(\widetilde{\DC} - \lambda I) = \det(S \DC^+ S^{-1} - \lambda SS^{-1}) = \det(S (\DC^+  -\lambda  I)S^{-1})= \det(\DC^+  -\lambda  I).
  \]
  It follows that all eigenvalues of $\widetilde{\DC}$ have positive real part. Analogously, the same assertion is true for the eigenvalues of $\widetilde{\FC}$ and $\FC$.

  Suppose now that $n=m$, then Theorem 7.1 in \cite{Robi99} assures that there exists a homeomorphism $\zeta: \R^n \rightarrow \R^n$ such that
  \[
    \zeta(\rme^{t\widetilde{\DC}}X)=\rme^{t\widetilde{\FC}}\zeta(X), \;\;\;\mbox{ for any }\;\;X \in \R^n.
  \]
  Defining $\xi:\fg^+\rightarrow\fh^+$ by $\xi  = T^{-1} \zeta \,S$ we see that
  \[
    \xi^{-1} \rme^{t\FC^+}\xi(X) =   S^{-1} \zeta^{-1}  \rme^{t(T\FC^+T^{-1})} \zeta S (X)
    =   S^{-1} \zeta^{-1}  \rme^{t\widetilde{\FC}} \zeta S (X)
  \]
  \[
     =  S^{-1} \rme^{t\widetilde{\DC}}  S (X)
     =   \rme^{t(S^{-1}\widetilde{\DC}S)}   (X)
     =   \rme^{t \DC^+}(X),
  \]
  which shows the topological conjugacy.

  Conversely, suppose that there exists a homeomorphism $\xi:\fg^+\rightarrow\fh^+$ such that
  \[
    \xi(\rme^{t\DC^+}X)=\rme^{t\FC^+}\xi(X), \;\;\;\mbox{ for any }\;\;X\in\fg^+.
  \]
  The map $\zeta: \R^n \rightarrow \R^m$ given by $\zeta(v) = T \xi \,S^{-1}(v)$ is certainly a homeomorphism which by the Invariance of Domain Theorem implies that $\dim\fg^+ = n = m =\dim \fh^+$ concluding the proof.
\end{proof}

\begin{theorem}\label{dimensionequality}
   The unstable subgroups of $\varphi_t$ and $\psi_t$ have the same dimension if and only if $\varphi^+_t$ and $\psi^+_t$ are conjugated. Analogously, the dimensions of their stable subgroups agree if and only if $\varphi^-_t$ and $\psi^-_t$ are conjugated.
\end{theorem}

\begin{proof}
  Let us do the unstable case, since the stable is analogous. By the above lemma, there exists $\xi:\fg^+\rightarrow\fh^+$ such that
  \[
    \xi(\rme^{t\DC^+}X)=\rme^{t\FC^+}\xi(X), \;\;\;\mbox{ for any }\;\;X\in\fg^+.
  \]
  By Proposition \ref{simplyconnected} the subgroups $G^+$ and $H^+$ are simply connected, which implies that the map
  \[
    \pi:G^+\rightarrow H^+, \;\;g\in G^+\mapsto \pi(g):=\exp_{H^+}(\xi(\exp_{G^+}^{-1}(g)))
  \]
  is well defined. Moreover, since both $\exp_{G^+}$ and $\exp_{H^+}$ are diffeomorphisms and $\xi$ is a homeomorphism, we have that $\pi$ is a homeomorphism. Let us show that $\pi$ conjugates $\varphi_t^+$ and $\psi_t^+$. Since $\varphi^+_t\circ\exp_{G^+}=\exp_{G^+}\circ\,\rme^{t\DC^+}$ and $\psi^+_t\circ\exp_{H^+}=\exp_{H^+}\circ\,\rme^{t\FC^+}$ we have, for any $g\in G^+$, that
  \[
    \pi(\varphi^+_t(g))=\exp_{H^+}(\xi(\exp_{G^+}^{-1}(\varphi^+_t(g))))=\exp_{H^+}(\xi(\rme^{t\DC^+}(\exp^{-1}_{G^+}(g))))
  \]
  \[
    =\exp_{H^+}(\rme^{t\FC^+}\xi(\exp_{G^+}^{-1}(g)))=\psi^+_t(\exp_{H^+}(\xi(\exp_{G^+}^{-1}(g))))=\psi^+_t(\pi(g))
  \]
  as desired.

  Conversely, suppose that there exists a homeomorphism $\pi:G^+\rightarrow H^+$ such that
  \[
    \pi( \varphi_t(g)) = \psi_t(\pi(g)), \;\;\;g\in G^+.
  \]
  Since $G^+$ and $H^+$ are connected, nilpotent and simply connected, it follows that the map $ \xi:\fg^{+} \rightarrow \fh^{+}$ given by
  \[
    \xi(X) = \exp_{H^+}^{-1}(\pi (\exp_{G^+}(X))), \;\;\;\mbox{ for any }\;\;X\in\fg^+
  \]
  is well defined and is in fact a homeomorphism between $\fg^+$ and $\fh^+$. Moreover, $\xi$ is a conjugacy between $\rme^{t\DC^+}$ and $\rme^{t\FC^+}$. In fact,
  \[
    \xi(\rme^{t\mathcal{D}^{+}}X)
     = \exp_{H^+}^{-1}(\pi (\exp_{G^+}(d\varphi_t^{+}X)))
     =  \exp_{H^+}^{-1}(\pi (\varphi_t^{+}(\exp_{G^+}(X))))
  \]
  \[
    = \exp_{H^+}^{-1}(\psi_t^{+}(\pi (\exp_{G^+}(X)))) = \rme^{t\mathcal{F}^{+}}( \exp_{H^+}^{-1}(\pi((\exp_{G^+}(X))))) = \rme^{t\mathcal{F}^{+}} \xi(X).
  \]
  Lemma \ref{lemadimensao} now assures that $\dim \fg^+ = \dim \fh^+$. It means that the unstable subgroup of $\varphi_t$ and $\psi_t$ have the same dimension.
\end{proof}

Now we have the main result of the paper.

\begin{theorem}\label{conjugatenecessary}
 Let us assume that $\XC$ and $\YC$ are hyperbolic. If the stable and unstable subgroups of $\varphi_t$ and $\psi_t$ have the same dimension, then $\varphi_t$ and $\psi_t$ are conjugated.
\end{theorem}

\begin{proof}
By the above theorem, there are homeomorphisms $\pi_u:G^+\rightarrow H^+$, that conjugates $\varphi^+_t$ and $\psi_t$, and $\pi_s:G^-\rightarrow H^-$ that conjugates $\varphi^-_t$ and $\psi^-_t$. By Proposition \ref{subgroups} we have, since $G$ is nilpotent, that $G=G^+G^-$ with $G^+\cap G^-=e_G$. Consequently, any $g\in G$ has a unique decomposition $g=g^+g^-$ with $g^+\in G^+$ and $g^-\in G^-$. Moreover, the same statement is true for $H$. Therefore, the map $\pi:G\rightarrow H$ given by
  \[
    g=g^+g^-\in G^+G^-\mapsto\pi(g):=\pi_u(g^+)\pi_s(g^-)\in H^+H^-=H
  \]
  is well defined and has inverse $\pi^{-1}(h^+h^-)=\pi_u^{-1}(h^+)\pi_s^{-1}(h^-)$. We will divide the rest of our proof in two steps:

  {\bf Step 1:} $\pi$ and $\pi^{-1}$ are continuous.
  Let us show the continuity of $\pi$ since the proof for $\pi^{-1}$ is analogous. Let then $(x_n)$ a sequence in $G$ and assume that $x_n\rightarrow x$. By Proposition \ref{subgroups} there are unique sequences $(g_n^+)$ in $G^+$ and $(g_n^-)$ in $G^-$ such that $x_n=g_n^+g_n^-$. If $x=g^+g^-$ we have by Lemma \ref{convergence} that $x_n\rightarrow x$ if and only if $g_n^{\pm}\rightarrow g^{\pm}$ in $G^{\pm}$. Since $\pi_u$ and $\pi_s$ are homeomorphism we have that $\pi_u(g_n^+)\rightarrow\pi_u(g^+)$ and $\pi_s(g_n^-)\rightarrow\pi_u(g^-)$ which again by Lemma \ref{convergence} now applied to $H$, implies that
  \[
    \pi(x_n)=\pi_u(g_n^+)\pi_s(g_n^-)\rightarrow \pi_u(g^+)\pi_u(g^-)=\pi(x)
  \]
  showing that $\pi$ is continuous.

  {\bf Step 2:} $\pi$ conjugates $\varphi_t$ and $\psi_t$;

  In fact
  \[
    \pi(\varphi_t(g))=\pi(\varphi_t(g^+)\varphi_t(g^-))=\pi_u(\varphi^+_t(g^+))\pi_s(\varphi_t^-(g^-))
  \]
  \[
   =\psi_t^+(\pi_u(g^+))\psi_t^-(\pi_s(g^-))=\psi_t(\pi_u(g^+))\psi_t(\pi_s(g^-))
  \]
  \[
  =\psi_t(\pi_u(g^+)\pi_s(g^-_1))=\psi_t(\pi(g)),
  \]
  for any $g\in G$, showing that $\pi$ conjugates $\varphi_t$ and $\psi_t$ and concluding the proof.
\end{proof}

\begin{corollary}
Let us assume that $\XC$ and $\YC$ are hyperbolic and that $G=H$. If the stable or the unstable subgroup of $\varphi_t$ and $\psi_t$ have the same dimension, then $\varphi_t$ and $\psi_t$ are conjugated.
\end{corollary}

\begin{proof}
In fact, if $G^+_1$ and $G_2^+$ are, respectively, the unstable subgroups of $\varphi_t$ and $\psi_t$ and, $G_1^-$ and $G_2^-$, respectively, their stable subgroups, then
$$\dim G_1^++\dim G_1^-=\dim G=\dim G_2^++\dim G^-_2$$
implying that $\dim G^+_1=\dim G^+_2$ if and only if $\dim G^-_1=\dim G_2^-$.
\end{proof}

\bigskip

We are now interested to prove the converse of Theorem \ref{conjugatenecessary}. Then we need the next result that shows that any conjugation between hyperbolic linear vector fields has to take the neutral element of $G$ to the neutral element of $H$.

\begin{lemma}\label{fixpoint}
  Let $\XC$ and $\YC$ be hyperbolic linear vector fields on $G$ and $H$, respectively. If the homeomorphism $\pi:G \rightarrow H$ conjugate $\varphi_t$ and $\psi_t$, then $\pi(e_G) =e_H$.
\end{lemma}
\begin{proof}
  We observe that $e_G \in G$ and $e_H \in H$ are, by item 5. of Proposition \ref{subgroups}, the unique fixed points of flows $\varphi_t$ and $\psi_t$, respectively. Then, the following equality
  \[
    \psi_t(\pi(e_G)) = \pi(\varphi_t(e_G)) = \pi(e_G).
  \]
  shows the Lemma.
\end{proof}

Now we are in conditions to prove the converse of Theorem \ref{conjugatenecessary}.

\begin{theorem} \label{conjugatesufficiency}  
  Let us assume that $\XC$ and $\YC$ are hyperbolic. If $\varphi_t$ and $\psi_t$ are conjugated, then their stable and unstable subgroups have the same dimension.
\end{theorem}
\begin{proof}
  Let $\pi:G \rightarrow H$ be a homeomorphism such that
  \[
    \pi(\varphi_t(g)) = \psi_t(\pi(g)).
  \]
 From Theorem \ref{dimensionequality} it is sufficient to show that $\varphi_t^{\pm}$ and $\psi_t^{\pm}$ are conjugated. We will only show that $\varphi_t^{-}$ and $\psi_t^{-}$ are conjugated, since the unstable case is analogous. We begin by showing that $\pi(G^-) = H^-$. Take $g \in G^{-}$. From Lemma \ref{fixpoint} and Theorem \ref{metriccharacterization} it follows that
  \begin{eqnarray*}
    e_H = \pi(e_G)=\pi\left(\lim_{t\to\infty}\varphi_t(g)\right)=\lim_{t \to \infty} \pi(\varphi_t(g)) = \lim_{t \to \infty}\psi_t(\pi(g)).
  \end{eqnarray*}
  Again by Theorem \ref{metriccharacterization} we get that $\pi(g) \in H^{-}$ showing that $\pi(G^{-}) \subset H^-$. Analogously we show that $\pi^{-1}(H^{-}) \subset G^-$ and consequently that $\pi(G^-)=H^-$. If we consider the restriction $\pi_s:=\pi|_{G^-}$ we have that $\pi_s$ is a homeomorphism between $G^-$ and $H^-$ and it certainly conjugates $\varphi_t^-$ and $\psi_t^-$ which from Theorem \ref{dimensionequality} implies that the stables subgroups of $\varphi_t$ and $\psi_t$ have the same dimension.
\end{proof}

\begin{remark}
  Someone can easily observe that Theorems \ref{conjugatenecessary} and \ref{conjugatesufficiency} are versions to Lie group $G$ of well known Theorems of topological conjugacy in $\mathbb{R}^n$( see for example section 4.7 in \cite{robinson}).
\end{remark}

\section{Lyapunov stability}

In this section we will show that the stability properties of a linear flow on a Lie group $G$ behaves in the same way as the one of the linear flow on the Lie algebra $\fg$ induced by the derivation $\DC$.

In order to characterize the stability, let us define the Lyapunov exponent at $g \in G$ in direction to $v \in T_gG$ by
\[
  \lambda(g,v)= \limsup_{t \to \infty} \frac{1}{t} \mathrm{log}(\|(d\varphi_t)_g(v)\|),
\]
where the norm $\| \cdot \|$ is given by the left invariant metric.

Our next step is to show the invariance of the Lyapunov exponent. In fact, since that $\varphi_t \circ L_g = L_{\varphi_t(g)} \circ \varphi_t$, it follows that
\[
  (d\varphi_t)_g(v) = (d\varphi_t)_g((dL_g)_e\circ (dL_{g^{-1}})_g(v)) = (dL_{\varphi_t(g)})_e\circ(d\varphi_t)_g((dL_{g^{-1}})_g(v)).
\]
As $\|\cdot \|$ is a left invariant norm we have that
\begin{eqnarray*}
  \lambda(g,v)
  & = & \limsup_{t \to \infty} \frac{1}{t} \log(\|(dL_{\varphi_t(g)})_e\circ(d\varphi_t)_g((dL_{g^{-1}})_g(v))\|)\\
  & = & \limsup_{t \to \infty} \frac{1}{t} \log(\|(d\varphi_t)_e((dL_{g^{-1}})_g(v))\|)\\
  & = & \lambda(e,(dL_{g^{-1}})_g(v)).
\end{eqnarray*}
It is clear that for $v \in \fg$ we obtain  $\lambda(g,v(g)) = \lambda(e,v(e))$. In other words, taking $v \in \fg$ we obtain
\[
  \lambda(e,v) = \limsup_{t \to \infty} \frac{1}{t} \log(\|\rme^{t\mathcal{D}}(v)\|).
\]

Our next Lemma is similar to well know result for Lyapunov exponent on $\mathbb{R}^n$.

\begin{lemma}
  Let $u, v \in \fg$, then $\lambda(e, u+v) \leq \max\{ \lambda(e,u), \lambda(e,v)\}$ and the equality is true if $\lambda(e,u) \neq \lambda(e,v)$.
\end{lemma}

This lemma will be used in the proof of forthcoming characterization of Lyapunov exponents. Before that we need to consider another decomposition of the Lie algebra $\fg$. Let us denote by $\lambda_1, \ldots, \lambda_k$ the $k$ distinct values in of the real parts of the derivation $\DC$. We have then that
$$
  \fg=\bigoplus_{i=1}^k\fg_{\lambda_i}, \;\;\mbox{ where }\;\;\fg_{\lambda_i}:=\bigoplus_{\alpha; \mathrm{Re}(\alpha)=\lambda_i}\fg_{\alpha}
$$

\begin{theorem}\label{lyapunovexponent}
It holds that
$$\lambda(e, v)= \lambda \;\;\;\Leftrightarrow\;\;\;v\in\fg_{\lambda}:=\bigoplus_{\alpha; \mathrm{Re}(\alpha)=\lambda}\fg_{\alpha}.$$
\end{theorem}
\begin{proof}
  We first suppose that $v \in \fg_{\lambda}$, then
  \begin{eqnarray*}
   \lambda(e, v)= \limsup_{t \to \infty} \frac{1}{t} \log(\|e^{t\DC}(v)\|)
    & = &  \limsup_{t \to \infty} \frac{1}{t} \log(\|e^{\alpha t}e^{(\DC-\alpha I)}(v)\|)\\
    & = & \mathrm{Re}(\alpha) + \limsup_{t \to \infty} \frac{1}{t} \log(\|e^{(\DC-\alpha I)}(v)\|).
  \end{eqnarray*}
  Since $\rme^{t(\DC-\alpha I)}(v) = \sum_{i=0}^d \frac{t^{i}}{i!}(\DC - \alpha I)^{i}(v)$ is a polynomial, it follows that
    \[
    \limsup_{t \to \infty} \frac{1}{t} \log(\|e^{(\DC-\alpha I)}(v)\|) = 0
  \]
 which gives us $\lambda(e,v) = \mathrm{Re}(\alpha)=\lambda$ as stated.

  Conversely, suppose that $\lambda(e,v) = \lambda$ and that $v \not\in \fg_{\lambda}$. Assume w.l.o.g. that $\lambda=\lambda_1$ and write $v = v_2 + v_3 + \ldots + v_k$ with $v_i \in \fg_{\lambda_i}$ for $i=2, \ldots, k$. Since $\lambda(e,v_i)=\lambda_i \neq\lambda_j= \lambda(e,v_j)$, for $ i, j= 1, \ldots, n$ with $i\neq j$, the above lemma assures that
  \begin{eqnarray*}
    \lambda_1
    & = & \lambda(e,v) = \lambda(e, v_2 + v_3 + \ldots + v_k)\\
    & = & \max\{\lambda(e,v_2), \lambda(e,v_3),\ldots, \lambda(e, v_k)\}\\
    & = & \max\{\lambda_2, \lambda_3,\ldots, \lambda_k\},
  \end{eqnarray*}
  where for the last equality we used the first part of Theorem. Since $\lambda_1\neq\lambda_i$ for $i=2, \ldots, k$ we have a contradiction.
\end{proof}

We follow by introducing the version of stability of the system (\ref{syst1}) on a Lie group for (see Definition 1.4.6 in \cite{coloniuskliemann}).

\begin{definition}
  Let $g\in G$ be a fixed point of $\mathcal{X}$. We say that $g$ is
  \begin{itemize}
    \item[1)]{\bf stable} if for all $g$-neighborhood $U$ there is a $g$-neighborhood $V$ such that $\varphi_t(V)\subset U$ for all $t\geq 0$;
    \item[2)] {\bf asymptotically stable} if it is stable and there exists a $g$-neighborhood $W$ such that $\lim_{t\rightarrow\infty}\varphi_t(x)=g$ whenever $x\in W$;
    \item[3)] {\bf exponentially stable} if there exist $c, \mu$ and a $g$-neighborhood $W$ such that for all $x\in W$ it holds that
    $$\varrho(\varphi_t(x), g)\leq c\rme^{-\mu t}\varrho(x, g),\;\;\;\;\mbox{ for all }\;\;t\geq 0;$$
    \item[4)] {\bf unstable} if it is not stable.
  \end{itemize}
\end{definition}

We should notice that, since property 3) is local, it does not depend on the metric that we choose on $G$. 

Next we prove a technical lemma that will be needed for the main results of this section.

\begin{lemma}
\label{conjugation}
Let $\XC$ and  $\YC$ be linear vector fields on the Lie groups $G$ and $H$, respectively, and $\pi:G\rightarrow H$ be a continuous map that commutates the linear flows of $\XC$ and $\YC$. If the fixed point $g$ of $\XC$ is stable (asymptotically stable) and there is a $g$-neighborhood $U$ such that $V=\pi(U)$ is open in $H$ and the restriction $\pi|_{U}$ is a homeomorphism, then the fixed point $\pi(g)$ of $\YC$ is stable (asymptotically stable). Moreover, if $\pi$ is a covering map the converse also holds.
\end{lemma}

\begin{proof}
Let us assume that $g$ is stable for $\XC$ and let $U'$ be a $\pi(g)$-neighborhood. By the property of $\pi$ around $g$, there exists a $g$-neighborhood $U$ such that $\pi$ restricted to $U$ is a homeomorphism and $\pi(U)\subset U'$. By the stability, there exists a $g$-neighborhood $V$ such that $\varphi_t(V)\subset U$ for all $t\geq 0$. Consequently $V'=\pi(V)$ is a $\pi(g)$-neighborhood and it holds that
$$\varphi'_t(V')=\varphi'_t(\pi(V))=\pi(\varphi_t(V))\subset\pi(U)\subset U', \;\;\;\mbox{ for all }\;\;t\geq 0$$
showing that $\pi(g)$ is stable for $\YC$.

If $g$ is asymptotically stable, there is a $g$-neighborhood $W$ such that $\lim_{t\rightarrow\infty}\varphi_t(x)=g$ for any $x\in W$. We can assume w.l.o.g. that $W$ is small enough such that $\pi$ restricted to $W$ is a homeomorphism. Then $W'=\pi(W)$ is a $\pi(g)$-neighborhood and
$$\lim_{t\rightarrow\infty}\varphi'_t(\pi(x))=\lim_{t\rightarrow\infty}\pi(\varphi_t(x))=\pi(\lim_{t\rightarrow\infty}\varphi_t(x))=\pi(g)$$
showing that $\pi(g)$ is asymptotically stable for $\YC$.

Let us assume now that $\pi$ is a covering map and that $\pi(g)$ is stable for $\YC$. Since $\pi$ is a covering map, there is a distinguished $\pi(g)$-neighborhood $U'$, that is, $\pi^{-1}(U')=\bigcup_{\alpha}U_{\alpha}$ is a disjoint union in $G$ such that $\pi$ restricted to each $U_{\alpha}$ is a homeomorphism onto $U'$. Let $U$ be a given $g$-neighborhood and assume w.l.o.g. that $U$ is the component of $\pi^{-1}(U')$ that contains $g$. By stability, there exists a $\pi(g)$-neighborhood $V'$ such that $\varphi'_t(V')\subset U'$ for all $t\geq 0$. Let $V\subset U$ be a $g$-neighborhood such that $\pi(V)\subset V'$. For $x\in V$ it holds that
$$\pi(\varphi_t(x))=\varphi'_t(\pi(x))\in\varphi_t'(V')\subset U', \;\;\;\mbox{ for all }\;\;t\geq 0$$
and consequently $\varphi_t(x)\in\pi^{-1}(U')$ for all $t\geq 0$. Since $\pi^{-1}(U')$ is a disjoint union and $x\in V\subset U$ we must have $\varphi_t(x)\in U$ for all $t\geq 0$. Being that $x\in V$ was arbitrary, we get that $\varphi_t(V)\subset U$ for all $t\geq 0$, showing that $g$ is stable for the linear vector field $\XC$.

The asymptotically stability follows, as above, from the fact that $\pi$ has a continuous local inverse.
\end{proof}

\bigskip
The following theorem characterizes, as for the Euclidian case, asymptotic and exponential stability at the identity $e\in G$ for a linear vector field in terms of the eigenvalues of $\DC$( see for instance Theorem 1.4.8 in \cite{coloniuskliemann}).

\begin{theorem}
For a linear vector field $\XC$  the following statements are equivalents:
  \begin{itemize}
    \item[(i)] The identity $e\in G$ is asymptotically stable;
    \item[(ii)] The identity $e\in G$ is exponentially stable;
    \item[(iii)] All Lyapunov exponents of $\varphi_t$ are negative;
    \item[(iv)] The stable subgroup $G^-$ satisfies $G=G^-$.
  \end{itemize}
\end{theorem}

\begin{proof}
Since $G=G^-$ if and only if $\fg=\fg^-$ we have that (iii) and (iv) are equivalent, by Theorem \ref{lyapunovexponent}. Moreover, by equation (\ref{metriccontract}) we have that (iii) and (iv) implies (ii) and (ii) certainly implies (i). We just need to show, for instance, that (i) implies (iv), which we will do in two steps:

{\bf Step 1:} If $e\in G$ is asymptotically stable, $G$ is nilpotent;

In fact, let $U$ be a neighborhood of $0\in\fg$ such that $\exp$ restricted to $U$ is a diffeomorphism and such that $\exp(U)\subset W$. For any $X\in\ker\DC$ let $\delta>0$ small enough such that $g=\exp(\delta X)\in W$. Since $\varphi_t(g)=g$ for any $t\in\R$ the asymptotic assumption implies that we must have $g=e$ and consequently that $X=0$ showing that $\ker\DC=\{0\}$. The derivation $\DC$ is then invertible which implies that $\fg$ is a nilpotent Lie algebra and so $G$ is a nilpotent Lie group.

{\bf Step 2:} If $e\in G$ is asymptotically stable, $G=G^-$.

The derivation $\DC$ on the Lie algebra $\fg$ can be identified with the linear vector field on $\fg$ given by $X\mapsto\DC(X)$. Its associated linear flow is given by $\rme^{t\DC}$. By the above step, $G$ is a nilpotent Lie group which implies that $\exp:\fg\rightarrow G$ is a covering map. Moreover, since $\varphi_t\circ\exp=\exp\circ\,\rme^{t\DC}$ we have that $e\in G$ is asymptotically stable if and only if $0\in\fg$ is asymptotically stable for the linear vector field induced by $\DC$.

By the results in \cite{coloniuskliemann} for linear Euclidian systems we have that $0\in\fg$ is asymptotically stable if and only if $\DC$ has only eigenvalues with negative real part, that is, $\fg=\fg^-$ implying that $G=G^-$ and concluding the proof.
\end{proof}

\begin{remark}
We should notice that the above result shows us that, as for linear Euclidian systems, local stability is equal to global stability. Moreover, in order for $e\in G$ be asymptotically stable for a linear vector field $\XC$ is necessary that for $G$ to be a simply connected nilpotent Lie group.
\end{remark}

The next result concerns the stability of a linear vector field.

\begin{theorem}
The identity $e\in G$ is stable for the linear vector field $\XC$ if $G=G^{-, 0}$ and $\DC$ restricted to $\fg^0$ is semisimple.
\end{theorem}

\begin{proof}
First we note that $G=G^{-, 0}$ if and only if $\fg=\fg^{-, 0}$. By Theorem 4.7 in \cite{coloniuskliemann} for linear Euclidian systems, the conditions that $\fg=\fg^{-, 0}$ and that $\DC|_{\fg^0}$ is semisimple is equivalent to $0\in\fg$ be stable for the linear vector field induced by $\DC$. Since $\exp$ is a local diffeomorphism around $0\in\fg$ we have by Lemma \ref{conjugation} that $0\in\fg$ stable for $\DC$ implies that $e\in G$ is stable for $\XC$ concluding the proof.
\end{proof}

The next result gives us a partial converse of the above theorem.

\begin{theorem}
If $e\in G$ is stable for the linear vector field $\XC$ then $G=G^{-, 0}$. Moreover, if $\exp_{G^0}:\fg^0\rightarrow G^0$ is a covering map then $e\in G$ stable for $\XC$ implies also that $\DC|_{\fg^0}$ is semisimple.
\end{theorem}

\begin{proof}
By equation (\ref{metricexpand}) the only element in $G^+$ that have bounded positive $\XC$-orbit is the identity. Therefore, if $e\in G$ is stable then $G^+=\{e\}$ and consequently $G=G^{-, 0}$.

Since $G^0$ is $\varphi$-invariant, the linear flow $\XC$ induces a linear vector field $\XC_{G^0}$ on $G^0$ such that the associated linear flow is the restriction $(\varphi_t)|_{G^0}$. Moreover, being that $G^0$ is a closed subgroup, it is not hard to prove that $e\in G$ stable for $\XC$ implies $e\in G^0$ stable for the restriction $\XC_{G^0}$.

If we assume that $\exp_{G^0}$ is a covering map, we have by Lemma \ref{conjugation} that $e\in G^0$ stable for $\XC_{G^0}$ if and only if $0\in\fg^0$ stable for $\DC|_{\fg^0}$ which by Theorem 4.7 of \cite{coloniuskliemann} implies that $\DC|_{\fg^0}$ is semisimple.
\end{proof}

When $G$ is a nilpotent Lie group the subgroup $G^0$ is also nilpotent and so the map $\exp_{G^0}:\mathfrak{g}^0\rightarrow G^0$ is a covering map. We have then the following.

\begin{corollary}
  If $G$ is a nilpotent Lie group then $e\in G$ is stable if and only if $G=G^{-, 0}$ and $\mathcal{D}|_{\mathfrak{g}^0}$ is semisimple.
\end{corollary}

\begin{remark}
  Another example where we have that the stability of the linear vector field $\mathcal{X}$ on the neutral element implies that $\mathcal{D}|_{\mathfrak{g}^0}$ is semisimple is when $G$ is a solvable Lie group and $\exp:\fg\rightarrow \widetilde{G}$ is a diffeomorphism, where $\widetilde{G}$ is the simply connected covering of $G$.
\end{remark}

\section{Conclusion}

We conclude by observing that our work is an initial step for several studies. We explain this assertion with two important problems. First, since now we understand the topological conjugacy of linear systems of type (\ref{syst1}), the next natural step is to study the topological conjugacy of linear control systems (\ref{contsyst1}). Second, to study the concept of Morse index of the linear system (\ref{syst1}). However, here, it is necessary to observe that by Corollary \ref{compacityproperty}, the compacity necessary on $G$ to this study implies that the flow of (\ref{syst1}) has no expanding or contracting subgroups. Consequently a suitable homogenous space has to be considered in order to study the Morse index. To conclude, the similarity of results make us believe that one can show that several results, founded in classical literature of linear system in $\mathbb{R}^n$, are still true for the linear systems (\ref{syst1}) on Lie groups.

\end{document}